\documentclass[11pt]{amsart}
\usepackage{latexsym,amsmath,amssymb,epsfig,amsthm}
\usepackage{graphicx}
\usepackage{tikz}
\usetikzlibrary{calc}
\usepackage[enableskew,vcentermath]{youngtab}
\usepackage{hyperref}
\hypersetup{colorlinks=false}
\input epsf
\newtheorem{theorem}{Theorem}[section]
\newtheorem{proposition}[theorem]{Proposition}
\newtheorem{corollary}[theorem]{Corollary}
\newtheorem{conjecture}[theorem]{Conjecture}
\newtheorem{definition}[theorem]{Definition}
\newtheorem{example}[theorem]{Example}
\newtheorem{lemma}[theorem]{Lemma}
\newtheorem{remark}[theorem]{Remark}

\newcommand{\asc}{{\rm asc}}

\newcommand{\bad}{{\rm bad}}

\newcommand{\des}{{\rm des}}

\newcommand{\exc}{{\rm exc}}
\newcommand{\fexc}{{\rm fexc}}
\newcommand{\fix}{{\rm fix}}
\newcommand{\Fix}{{\rm Fix}}

\newcommand{\Orb}{{\rm Orb}}
\newcommand{\sd}{{\rm sd}}

\newcommand{\aA}{{\mathcal A}}
\newcommand{\bB}{{\mathcal B}}
\newcommand{\dD}{{\mathcal D}}
\newcommand{\eE}{{\mathcal E}}
\newcommand{\fF}{{\mathcal F}}
\newcommand{\hH}{{\mathcal H}}
\newcommand{\iI}{{\mathcal I}}

\newcommand{\pP}{{\mathcal P}}
\newcommand{\lL}{{\mathcal L}}

\newcommand{\RR}{{\mathbb R}}
\newcommand{\fS}{{\mathfrak S}}
\newcommand{\NN}{{\mathbb N}}
\newcommand{\ZZ}{{\mathbb Z}}

\renewcommand{\to}{\rightarrow}

\newcommand{\sm}{{\smallsetminus}}
\begin{document}
\title[Real-rootedness of the Eulerian 
transformation]
{On the real-rootedness of the Eulerian 
transformation}

\author{Christos~A.~Athanasiadis}

\address{Department of Mathematics\\
National and Kapodistrian University of Athens\\
Panepistimioupolis\\ 15784 Athens, Greece}
\email{caath@math.uoa.gr}

%\date{May 16, 2023}
\thanks{Research supported by the Hellenic 
Foundation for Research and Innovation (H.F.R.I.) 
under the `2nd Call for H.F.R.I. Research Projects 
to support Faculty Members \& Researchers' 
(Project Number: HFRI-FM20-04537).}

\begin{abstract}
The Eulerian transformation is the linear operator 
on polynomials in one variable with real coefficients 
which maps the powers of this variable to the 
corresponding Eulerian polynomials. The derangement 
transformation is defined similarly. Br\"and\'en and 
Jochemko have 
conjectured that the Eulerian transforms of a 
class of polynomials with nonnegative coefficients, 
which includes those having all their roots in the 
interval $[-1,0]$, have only real zeros. This 
conjecture is proven in this paper. More general 
transformations are introduced in the 
combinatorial-geometric context of uniform 
triangulations of simplicial complexes, where 
Eulerian and derangement transformations 
arise in the special case of barycentric 
subdivision, and are shown to have strong 
unimodality and gamma-positivity properties. 
General real-rootedness conjectures for these
transformations, which unify various results 
and conjectures in the literature, are also 
proposed.
\end{abstract}

\maketitle

\section{Introduction}
\label{sec:intro}
 
Eulerian polynomials form one of the most important 
and well studied families of polynomials in 
mathematics, playing a prominent role in combinatorics
and elsewhere; see, for instance, \cite{FSc70, Pet15} 
\cite[Section~1.4]{StaEC1}. The $n$th Eulerian 
polynomial is defined as
\[ A_n(x) \, = \sum_{w \in \fS_n} x^{\des(w)} 
          \, = \sum_{w \in \fS_n} x^{\exc(w)}, \]
where $\fS_n$ is the symmetric group of permutations 
of the set $[n] := \{1, 2,\dots,n\}$ and 
\begin{align*}
\des(w) & = \# \{ i \in [n-1] : w(i) > w(i+1) \} \\
\exc(w) & = \# \{ i \in [n-1] : w(i) > i \}
\end{align*}
is the descent and excedance number of $w \in \fS_n$,
respectively, with the convention $A_0(x) := 1$.
A well known result, often attributed to 
Frobenius~\cite{Fr10}, states that $A_n(x)$ has only 
real roots. Far less is known about the 
real-rootedness of linear combinations of Eulerian
polynomials. Following the historical approach to 
study linear transformations which preserve 
real-rootedness properties (see the discussions in 
 \cite[Section~7.7]{Bra15} \cite[Section~1]{BJ22}),
Br\"and\'en and Jochemko~\cite{BJ22} considered 
the linear operator $\aA^\circ: \RR[x] \to \RR[x]$, 
defined by setting

\begin{equation}  
\aA^\circ(x^n) = \begin{cases}
    1, & \text{if $n=0$} \\
    xA_n(x), & \text{if $n \ge 1$} \end{cases}  
\label{eq:def-eulerian}
\end{equation} 
for $n \in \NN := \{0, 1, 2,\dots\}$, and named 
it the \textit{Eulerian transformation} (the 
slight difference from~\cite{BJ22} in the notation 
adopted here will be explained by our discussion 
in the sequel). They disproved a conjecture of 
Brenti~\cite[p.~32]{Bre89}, stating that 
$\aA^\circ$ preserves the class of polynomials 
having only real and nonpositive roots, and 
conjectured~\cite[Conjecture~1]{BJ22} that 
$\aA^\circ(p(x))$ has only real roots for every 
polynomial $p(x)$ which can be written as a 
nonnegative linear combination of the 
polynomials $x^{n-k} (1+x)^k$ for $k \in \{0, 
1,\dots,n\}$ for some $n \in \NN$. This class, 
which we will denote by $\pP_n[x]$ throughout 
this paper, contains all polynomials of degree 
$n$ with nonnegative coefficients which have 
all their roots in the interval $[-1,0]$. To 
support their conjecture, Br\"and\'en and 
Jochemko showed~\cite[Theorem~3.2]{BJ22} 
that $\aA^\circ(p(x))$ has a unimodal symmetric 
decomposition with respect to $n$ for every 
$p(x) \in \pP_n[x]$; in particular, 
$\aA^\circ(p(x))$ is unimodal, 
with a peak at position $\lceil n/2 \rceil$ 
(undefined terminology on polynomials is 
explained in Section~\ref{sec:pre}). They also 
studied other subtle properties of the Eulerian 
transformation. 

This paper aims to prove 
\cite[Conjecture~1]{BJ22}, provide new 
combinatorial and geometric interpretations of 
the Eulerian transformation and place the latter 
in a much broader and natural 
combinatorial-geometric context. Our first main  
result proves the aforementioned conjecture 
of~\cite{BJ22} and strengthens 
\cite[Theorem~3.2]{BJ22}. The polynomial 
$\aA^\circ((1+x)^n)$ which appears in the 
statement is the $n$th binomial Eulerian 
polynomial; see Section~\ref{sec:euler} for 
more information.
\begin{theorem} \label{thm:main} 
For every $p(x) \in \pP_n[x]$:
\begin{itemize}
\itemsep=0pt
\item[(a)] 
The polynomial $\aA^\circ(p(x))$ has only real 
roots. Moreover, it interlaces $xA_n(x)$ and it 
is interlaced by $\aA^\circ((1+x)^n)$. 

\item[(b)] 
The polynomial $\aA^\circ(p(x))$ has a real-rooted 
and interlacing (in particular, unimodal and 
$\gamma$-positive) symmetric decomposition with 
respect to $n$.

\end{itemize}
\end{theorem}

For a simple nontrivial application to linear 
combinations of $A_n(x), A_{n-1}(x)$ and 
$A_{n-2}(x)$, see Example~\ref{ex:simple-cor}. 

The unimodality, $\gamma$-positivity and 
real-rootedness properties of 
Theorem~\ref{thm:main} do not apply exclusively 
to the Eulerian transformation. We propose one 
way to generalize these results as follows. Our 
motivation 
comes from the fact that the right-hand side of 
Equation~(\ref{eq:def-eulerian}) can be interpreted
as the interior $h$-polynomial of the barycentric 
subdivision of the $(n-1)$-dimensional simplex, 
denoted here by $\sigma_n$ (basic definitions and 
terminology on the face enumeration of simplicial 
complexes are explained in Section~\ref{sec:geom}). 
It seems natural to inquire about the behavior of 
the transformation obtained when barycentric 
subdivision is replaced by other types of 
triangulations. 

A suitable context 
is provided by the theory of uniform triangulations,
developed in~\cite{Ath22}. A triangulation $\Gamma$ 
of the simplex $\sigma_n$ is called \emph{uniform} 
if the $f$-vector of the restriction of $\Gamma$ to
a face of $\sigma_n$ depends only on the dimension 
of that face. One can then consider the 
$h$-polynomial $h_\fF (\sigma_m,x)$ and the interior 
$h$-polynomial $h^\circ_\fF(\sigma_m,x)$ of the 
restriction of $\Gamma$ to any $(m-1)$-dimensional 
face of $\sigma_n$ (where $\fF$ records some 
combinatorial data determined by $\Gamma$). We 
define the linear operator $\hH^\circ_\fF: \RR_n[x] 
\to \RR_n[x]$ by setting
\[ \hH^\circ_\fF (x^m) = h^\circ_\fF (\sigma_m,x) 
   \] 
for $m \in \{0, 1,\dots,n\}$, where $\RR_n[x]$ 
stands for the space of polynomials of degree 
at most $n$ with real coefficients. This map 
reduces to the restriction of the Eulerian 
transformation $\aA^\circ$ to $\RR_n[x]$ in the 
special case that $\Gamma$ is the barycentric 
subdivision of $\sigma_n$.
Following~\cite{Ath22+}, we say that a 
triangulation $\Gamma$ of the simplex $\sigma_n 
= 2^V$ is \emph{theta unimodal} (respectively, 
\emph{theta $\gamma$-positive}) if 
$\theta(\Gamma_F,x) := h(\Gamma_F,x) - h(\partial 
\Gamma_F,x)$ is unimodal (respectively, 
$\gamma$-positive) for every $F \subseteq V$, 
where $\Gamma_F$ is the restriction of $\Gamma$ 
to the face $F$ of $\sigma_n$. For uniform 
triangulations, this means that 
\[ \theta_\fF(\sigma_m,x) := h_\fF (\sigma_m,x) -
   h_\fF (\partial \sigma_m,x) \]
is unimodal (respectively, $\gamma$-positive) for 
every $m \in \{0, 1,\dots,n\}$, where $h_\fF 
(\partial \sigma_m,x)$ stands for the 
$h$-polynomial of the restriction of $\Gamma$ to 
the boundary complex of any $(m-1)$-dimensional 
face of $\sigma_n$. As discussed in 
\cite[Section~5]{Ath22+} (see 
\cite[Theorem~5.1]{Ath22+}), it follows from 
\cite[Theorem~50]{AY20} and the $g$-theorem for 
triangulations of spheres \cite{Ad18, APP21, PP20} 
that $\Gamma$ is theta unimodal whenever 
$\partial \Gamma_F$ is a 
vertex-induced subcomplex of $\Gamma_F$ for every 
$F \subseteq V$. Moreover, several classes of 
triangulations of the simplex with interesting 
enumerative combinatorics are known to be theta 
$\gamma$-positive (see Section~\ref{sec:uniform}).
For the barycentric subdivision $\theta_\fF
(\sigma_m,x)$ is identically zero, and hence 
trivially $\gamma$-positive, for $m \ge 1$. 

Our second main result significantly generalizes 
\cite[Theorem~3.2]{BJ22} and supports a more 
general conjecture than \cite[Conjecture~1]{BJ22} 
which can be stated in this context (see 
Conjecture~\ref{conj:main}).
\begin{theorem} \label{thm:gamma} 
Let $\Gamma$ be a uniform triangulation of the 
$(n-1)$-dimensional simplex. 
\begin{itemize}
\itemsep=0pt
\item[(a)] 
If $\Gamma$ is theta unimodal, then 
$\hH_\fF^\circ(p(x))$ has a unimodal symmetric 
decomposition with respect to $n$ for every $p(x) 
\in \pP_n[x]$.

\item[(b)] 
If $\Gamma$ is theta 
$\gamma$-positive, then $\hH_\fF^\circ(p(x))$ 
has a $\gamma$-positive symmetric decomposition 
with respect to $n$ for every $p(x) \in \pP_n[x]$.

\end{itemize}
\end{theorem}

The structure and other results of this paper can 
be described as follows. Section~\ref{sec:pre}
provides basic definitions and background 
on real polynomials and their roots.
Theorem~\ref{thm:main} is proven in 
Section~\ref{sec:euler}, after the combinatorics 
of the polynomials $\aA^\circ 
(x^{n-k}(1+x)^k)$ is sufficiently developed. The 
method extends that followed by Br\"and\'en and 
Jochemko in~\cite[Section~4]{BJ22} to obtain 
various partial results (although this is only 
one out of several possible approaches one may 
follow to attack \cite[Conjecture~1]{BJ22}).  
Section~\ref{sec:derange} is concerned with a 
close relative of the Eulerian transformation, 
introduced and studied by Br\"and\'en and 
Solus~\cite[Section~3.2]{BS21}, termed here as 
the derangement transformation. 
Proposition~\ref{prop:dnk} provides a common 
generalization of the gamma-positivity of Eulerian 
and derangement polynomials; a result which easily 
implies Theorem~\ref{thm:gamma} in the crucial 
special case of the barycentric subdivision. An 
analogue of Theorem~\ref{thm:main} (b) for the 
derangement transformation 
(Corollary~\ref{thm:derange}) is also derived in 
Section~\ref{sec:derange} from 
Proposition~\ref{prop:dnk} and results 
of \cite[Section~3.2]{BS21}. 

Section~\ref{sec:geom} reviews basic enumerative 
combinatorics of triangulations (and uniform 
triangulations, in particular) of simplicial 
complexes. Given a uniform triangulation of the
simplex $\sigma_n$, it introduces linear 
transformations $\hH_\fF^\circ$ and $\lL_\fF$ 
which reduce to the Eulerian and derangement 
transformations, respectively, in the special 
case of barycentric subdivision and provides a
geometric interpretation of the polynomials 
$\hH_\fF^\circ (x^{n-k}(1+x)^k)$ (see 
Proposition~\ref{prop:geom}). 
Theorem~\ref{thm:gamma} and its analogue for 
$\lL_\fF$ (Corollary~\ref{cor:gamma}) is 
proven in Section~\ref{sec:uniform}. Essential 
ingredients of the proof are a family of 
polynomials interpolating between the 
$h$-polynomial and the local $h$-polynomial 
of a triangulation of a simplex (see 
Definition~\ref{def:local-VE}), properties of 
theta polynomials established in~\cite{Ath22+} 
and Proposition~\ref{prop:dnk}. 
Section~\ref{sec:roots} proposes 
generalizations of \cite[Conjecture~1]{BJ22} 
for the transformations $\hH_\fF^\circ$ and 
$\lL_\fF$ and discusses supporting evidence
and some further directions.

\section{Recollections of polynomials and their 
         roots}
\label{sec:pre}

This section explains basic background and terminology 
on real polynomials which will be useful in the 
following sections.
We recall that $\RR_n[x]$ stands for the space of 
polynomials $p(x) \in \RR[x]$ of degree at most $n$.
We will denote by $\iI_n(p(x))$ the reciprocal 
$x^n p(1/x)$ with respect to $n$ of a polynomial 
$p(x) \in \RR_n[x]$. A polynomial $p(x) = a_0 + a_1 x 
+ \cdots + a_n x^n \in \RR_n[x]$ is called
\begin{itemize}
\item[$\bullet$] 
  \emph{symmetric}, with center of symmetry $n/2$, if 
	$a_i = a_{n-i}$ for all $0 \le i \le n$,
\item[$\bullet$] 
  \emph{unimodal}, with a peak at position $k$, if $0 
	\le a_0 \le a_1 \le \cdots \le a_k \ge a_{k+1} \ge 
	\cdots \ge a_n \ge 0$,
\item[$\bullet$] 
  \emph{$\gamma$-positive}, with center of symmetry $n/2$,
	if $p(x) = \sum_{i=0}^{\lfloor n/2 \rfloor} \gamma_i 
	x^i (1+x)^{n-2i}$ for some nonnegative real numbers 
	$\gamma_0, \gamma_1,\dots,\gamma_{\lfloor n/2 \rfloor}$,
\item[$\bullet$] 
  \emph{real-rooted}, if every root of $p(x)$ is real, or 
	$p(x) \equiv 0$.
\end{itemize}
Every $\gamma$-positive polynomial is symmetric and 
unimodal and every real-rooted and symmetric polynomial
with nonnegative coefficients is $\gamma$-positive; see
\cite{Ath18, Bra15, Ga05, Sta89} \cite[Chapter~4]{Pet15} 
for more information on the connections among these 
concepts. Note that unimodal polynomials are assumed to
have nonnegative coefficients in this paper.

A real-rooted polynomial $p(x)$, with 
roots $\alpha_1 \ge \alpha_2 \ge \cdots$, is said to 
\emph{interlace} a real-rooted polynomial $q(x)$, with 
roots $\beta_1 \ge \beta_2 \ge \cdots$, if
\[ \cdots \le \alpha_2 \le \beta_2 \le \alpha_1 \le
   \beta_1. \]
We then write $p(x) \preceq q(x)$. By convention, 
the zero polynomial interlaces and is interlaced by 
every real-rooted polynomial and nonzero constant 
polynomials strictly interlace all polynomials of 
degree at most one. For real-rooted polynomials 
$p(x), q(x) \in \RR_n[x]$ with nonnegative 
coefficients, $p(x) \preceq q(x) \Rightarrow q(x) 
\preceq xp(x)$ and $p(x) \preceq q(x) \Rightarrow 
\iI_n(q(x)) \preceq \iI_n(p(x))$. If 
two or more real-rooted polynomials with positive 
leading coefficients interlace (respectively, are 
interlaced by) a real-rooted polynomial $p(x)$, 
then so does their sum. We will use these properties 
of interlacing in Section~\ref{sec:euler}; standard
references are~\cite[Section~7.8]{Bra15} 
\cite{Fi06}.  

A sequence $(p_0(x), p_1(x),\dots,p_m(x))$ of
real-rooted polynomials is called 
\emph{interlacing} if $p_i(x) \preceq p_j(x)$ for 
$0 \le i < j \le m$. The following 
lemmas will be applied in Section~\ref{sec:euler}. 
Part (b) of Lemma~\ref{lem:interlace-rec} is a 
special case of \cite[Theorem~2.4]{HZ19} (another 
version appears in the proof of 
\cite[Theorem~4.4]{BJ22}).
\begin{lemma} \label{lem:interlace-trans}
{\rm (\cite[Proposition~3.3]{Wa92})}
Let $p_0(x), p_1(x),\dots,p_m(x)$ be nonzero
real-rooted polynomials. If $p_{i-1}(x) \preceq 
p_i(x)$ for every $i \in [m]$ and $p_0(x) \preceq 
p_m(x)$, then $(p_0(x), p_1(x),\dots,p_m(x))$ is an 
interlacing sequence.
\end{lemma} 

\begin{lemma} \label{lem:interlace-rec} 
Let $(p_0(x), p_1(x),\dots,p_m(x))$ be an interlacing 
sequence of real-rooted polynomials with positive 
leading coefficients.
\begin{itemize}
\itemsep=0pt
\item[(a)]
Every nonnegative linear combination $p(x)$ of 
$p_0(x), p_1(x),\dots,p_m(x)$ is real-rooted. Moreover, 
$p_0(x) \preceq p(x) \preceq p_m(x)$. 

\item[(b)]
{\rm (cf. \cite[Theorem~2.4]{HZ19})}
The sequence $(q_0(x), q_1(x),\dots,q_{m+1}(x))$ 
defined by
\[ q_0(x) = \alpha(x) p_0(x) + \sum_{i=1}^m p_i(x), 
   \]
where $\alpha(x) = 1$ or $\alpha(x) = 1+x$, and
\[ q_k(x) = x \sum_{i=0}^{k-1} p_i(x) + 
             \sum_{i=k}^m p_i(x) \]
for $k \in \{1, 2,\dots,m+1\}$ is also interlacing.

\end{itemize}
\end{lemma}

\begin{example} \label{ex:pnk} \rm 
Let
\begin{align} 
p_{n,k}(x) & = \sum_{w \in \fS_{n+1} : \, w(1) = k+1} 
                  x^{\des(w)} = 
							 \sum_{w \in \fS_{n+1} : \, w(n+1) = k+1} 
                  x^{\asc(w)} 	\label{eq:pnk-1} \\
           & = \sum_{w \in \fS_{n+1}: \, w^{-1}(1) = 
					           k+1} x^{\exc(w)} \label{eq:pnk-2}
\end{align}
for $k \in \{0, 1,\dots,n\}$, so that $p_{n,0}(x) = 
A_n(x)$ and $p_{n,n}(x) = xA_n(x)$. The polynomials 
$p_{n,k}(x)$ appeared in the work of Brenti and 
Welker~\cite{BW08} on $f$-vectors of barycentric 
subdivisions and, independently, in 
\cite[Section~2.2]{CGSW07}, and have been studied 
intensely since then; see \cite[Section~2]{Ath22+} 
\cite[Section~3]{BJM19} \cite[Example~7.8.8]{Bra15} 
and the references given there. They satisfy the 
recurrence 
\begin{equation} \label{eq:pnk-rec}
p_{n,k}(x) = x \sum_{i=0}^{k-1} p_{n-1,i}(x) + 
             \sum_{i=k}^{n-1} p_{n-1,i}(x) 
\end{equation}
for $k \in \{0, 1,\dots,n\}$. Hence, by 
Lemma~\ref{lem:interlace-rec} (b), 
$(p_{n,k}(x))_{0 \le k \le n}$ is an interlacing 
sequence of real-rooted polynomials for every $n \in 
\NN$.
\qed
\end{example}

Given $p(x) \in \RR_n[x]$, there exist unique 
symmetric polynomials $a(x) \in \RR_n[x]$ and $b(x)
\in \RR_{n-1}[x]$ with centers of symmetry $n/2$ and 
$(n-1)/2$, respectively, such that $p(x) = a(x) + 
xb(x)$. This expression is known as the 
\emph{symmetric decomposition} (or 
\emph{Stapledon decomposition}) of $p(x)$ with 
respect to $n$. Then, $p(x)$ is said to have a 
\emph{nonnegative} (respectively, \emph{unimodal}, 
\emph{$\gamma$-positive} or \emph{real-rooted}) 
\emph{symmetric decomposition} with respect to $n$ 
if $a(x)$ and $b(x)$ have nonnegative coefficients 
(respectively, are unimodal, $\gamma$-positive or
real-rooted); see \cite[Section~5]{Ath18} 
\cite{AT21, BS21} for more information about 
these concepts. A nonnegative, real-rooted symmetric 
decomposition $p(x) = a(x) + xb(x)$ is said to be 
\emph{interlacing} if $b(x) \preceq a(x)$. Then 
\cite[Theorem~2.7]{BS21}, $p(x)$ is 
real-rooted and is interlaced by each one of 
$a(x)$, $b(x)$ and $\iI_n (p(x))$. Moreover, 
$b(x) \preceq a(x) \Leftrightarrow a(x) \preceq 
p(x) \Leftrightarrow b(x) \preceq p(x) 
\Leftrightarrow \iI_n (p(x)) \preceq p(x)$, under 
the assumption that $p(x)$ has a nonnegative
symmetric decomposition with respect to $n$.

\section{Combinatorics of the Eulerian transformation}
\label{sec:euler}

This section studies the combinatorics of the 
Eulerian transformation and, especially, the 
polynomials $\aA^\circ (x^{n-k}(1+x)^k)$ for $k \in 
\{0, 1,\dots,n\}$, and proves Theorem~\ref{thm:main}. 
For convenience, we consider the polynomials 
$q_{n,k}(x)$ defined by
\begin{equation}  
\iI_n (q_{n,k}(x)) = \aA^\circ 
                        \left( x^{n-k}(1+x)^k \right)  
\label{eq:def-qnk}
\end{equation} 
for $k \in \{0, 1,\dots,n\}$ instead. They are
shown on Table~\ref{tab:qnk} for $n \le 4$.
We have $q_{n,0}(x) = A_n(x)$ and $q_{n,n}(x) 
= \widetilde{A}_n(x)$ for every $n \in \NN$, where
\[ \widetilde{A}_n(x) := \aA^\circ((1+x)^n) = 1 + x 
   \sum_{i=1}^n 
   {n \choose i} A_i(x) = \sum_{i=0}^n {n \choose i}
	 x^{n-i} A_i(x) \]
is the $n$th \emph{binomial Eulerian polynomial}.
This polynomial was shown to be symmetric and 
$\gamma$-positive, with center of symmetry $n/2$, 
in~\cite[Section~11]{PRW08} (see also 
\cite[Section~2.1]{Ath18} \cite[Section~3]{BJ22} 
\cite{SW20}) and real-rooted 
in~\cite[Section~3]{HZ19} (see also 
\cite[Section~4]{BJ22}). Moreover, as a special 
case of \cite[Theorem~4.5]{BJ22}, $A_n(x)$
interlaces $\widetilde{A}_n(x)$.
We will show that each $q_{n,k}(x)$ is real-rooted
and that the sequence 
$(q_{n,k}(x))_{0 \le k \le n}$ is interlacing 
for every $n \in \NN$.

{\scriptsize
\begin{table}[hptb]
\begin{center}
\begin{tabular}{| l || l | l | l | l | l ||} \hline
& $k=0$ & $k=1$ & $k=2$ & $k=3$ \\ \hline \hline
$n=0$   & 1 &  & & \\ 
         \hline
$n=1$   & 1 & $1+x$ & & \\ \hline
$n=2$  & $1+x$ & $1+2x$ & $1+3x+x^2$ & \\ \hline
$n=3$  & $1+4x+x^2$ & $1+5x+2x^2$ & $1+6x+4x^2$ & 
         $1+7x+7x^2+x^3$ \\ \hline
$n=4$  & $1+11x+11x^2+x^3$ & $1+12x+15x^2+2x^3$ & 
         $1+13x+20x^2+4x^3$ & $1+14x+26x^2+8x^3$ 
				 \\ \hline
\end{tabular}

\bigskip
\begin{tabular}{| l || l ||} \hline
& $k=4$ \\ \hline \hline
$n=0$  & \\ \hline
$n=1$  & \\ \hline
$n=2$  & \\ \hline
$n=3$  & \\ \hline
$n=4$  & $1+15x+33x^2+15x^3+x^4$ \\ \hline
\end{tabular}

\bigskip
\caption{The polynomials $q_{n,k}(x)$ for $n \le 4$.}
\label{tab:qnk}
\end{center}
\end{table}}

The following statement lists some recurrences 
and combinatorial formulas for the polynomials 
$q_{n,k}(x)$; a geometric interpretation will be 
given in Section~\ref{sec:geom}. 

\begin{proposition} \label{prop:qnk} 
\begin{itemize}
\itemsep=0pt
\item[{\rm (a)}]
The polynomials $q_{n,k}(x)$ satisfy the recurrence 
\[ q_{n,k+1}(x) = q_{n,k}(x) + xq_{n-1,k}(x) \] 
for $0 \le k < n$.

\item[{\rm (b)}]
We have
\begin{align} 
q_{n,k}(x) & = \sum_{i=0}^k {k \choose i} x^i 
               A_{n-i}(x) \label{eq:qnk-1} \\
           & = \sum_{w \in \fS_n} (1+x)^{\fix_k(w)} 
					     x^{\exc(w)} \label{eq:qnk-2} \\
           & = \sum_{i=0}^k {k \choose i}
               p_{n-i,k-i}(x) \label{eq:qnk-3}
\end{align}
for $0 \le k \le n$, where $\fix_k(w)$ is the number 
of fixed points of $w \in \fS_n$ not exceeding $k$
and $p_{n,k}(x)$ are the polynomials of 
Example~\ref{ex:pnk}.
\end{itemize}
\end{proposition}

\begin{proof}
For part (a) we note that $x^{n-k-1}(1+x)^{k+1} = 
x^{n-k}(1+x)^k + x^{n-k-1}(1+x)^k$, which implies 
that 

\begin{align*} 
x^n q_{n,k+1}(1/x) & = \aA^\circ 
    \left( x^{n-k-1}(1+x)^{k+1} \right) = 
		\aA^\circ \left( x^{n-k}(1+x)^k \right) + \,
		\aA^\circ \left( x^{n-k-1}(1+x)^k \right) 
\\ & = x^n q_{n,k}(1/x) + x^{n-1} q_{n-1,k}(1/x).
\end{align*} 

\medskip
Equation~(\ref{eq:qnk-1}) follows from the 
recurrence of part (a) by induction on $k$, or by 
expanding the binomial $(1+x)^k$ in the defining 
Equation~(\ref{eq:def-qnk}) for $q_{n,k}(x)$ and 
applying the linearity of $\aA^\circ$. 

To derive 
Equation~(\ref{eq:qnk-2}) from~(\ref{eq:qnk-1}) we 
consider (as in \cite[Section~3]{BJ22}) decorated 
permutations of $[n]$, meaning that some of their 
fixed points (possibly all or none) are colored 
black. We denote by $\fS_n^\ast$ the set of decorated 
permutations of $[n]$ and for $w \in \fS_n^\ast$
we denote by $\fix^\ast(w)$ and $\Fix^\ast(w)$ the 
number and the set of fixed points of $w$ which are 
colored black, respectively, and by $\exc(w)$ the 
excedance number of the (undecorated) permutation 
which corresponds to $w$. Then,  

\begin{align*} 
\sum_{w \in \fS_n} (1+x)^{\fix_k(w)} x^{\exc(w)} 
& = \sum_{w \in \fS_n^\ast : \, \Fix^\ast(w) 
    \subseteq [k]} x^{\fix^\ast(w) + \exc(w)} \\ & =  
		\sum_{i=0}^k {k \choose i} x^i 
	  \sum_{u \in \fS_{n-i}} x^{\exc(u)} \ = \ 
	\sum_{i=0}^k {k \choose i} x^i A_{n-i}(x).
\end{align*} 

To derive Equation~(\ref{eq:qnk-3}) 
from Equation~(\ref{eq:qnk-1}), we expand $x^i = 
\sum_{j=0}^i {i \choose j} (x-1)^{i-j}$ on the 
right-hand side of~(\ref{eq:qnk-1}), 
apply the identity ${k \choose i} {i \choose j} 
= {k \choose j} {k-j \choose i-j}$, change the 
order of summation and use the formula
\[ p_{n,k}(x) = \sum_{i=0}^k {k \choose i} 
   (x-1)^i A_{n-i}(x), \]
which is easily derived from the identity
\cite[Equation~(4)]{BW08}
\[ \sum_{m \ge 0} m^k (1+m)^{n-k} x^m = 
   \frac{p_{n,k}(x)}{(1-x)^{n+1}}. \]
We omit the details, which are straightforward.
\end{proof}

To prove the real-rootedness of $q_{n,k}(x)$, we
introduce the polynomials
\begin{equation} \label{eq:def-qnmk}
q_{n,k,j}(x) = \sum_{w \in \fS_{n+1}: \, 
               w^{-1}(1) = j+1} (1+x)^{\fix_k(w)} 
					     x^{\exc(w)} 
\end{equation}
for $k \in \{0, 1,\dots,n+1\}$ and $j \in 
\{0, 1\dots,n\}$, where $\fix_k(w)$ is as in 
Proposition~\ref{prop:qnk}. An application of 
the fundamental transformation 
\cite[Section~I.3]{FSc70} 
\cite[Section~1.3]{StaEC1}, obtained when 
each cycle of a permutation $w \in \fS_{n+1}$ 
is written with its smallest element last and 
cycles are arranged in increasing order of their
smallest element, yields that
\begin{align} \label{eq:def-qnmk-alt}
q_{n,k,j}(x) & = \sum_{w \in \fS_{n+1}: \, 
               w(1) = j+1} (1+x)^{\bad_k(w)} 
					     x^{\des(w)} \\
							\label{eq:qnk5}
q_{n,k}(x) &   = \sum_{w \in \fS_n} 
               (1+x)^{\bad_k(w)} x^{\des(w)}						
\end{align}
where, for $w \in \fS_n$, $\bad_k(w)$ is the 
number of indices $i \in [n]$ for which $w(i) \le 
k$ is a right-to-left minimum of $w$ (meaning that 
$w(i) \le w(j)$ for $i \le j \le n$) and either 
$i = 1$ or $w(i-1) < w(i)$. We also set 
\begin{equation} \label{eq:def-q*nmk}
q^\ast_{n,k,j}(x) = \begin{cases} {\displaystyle 
    \frac{ q_{n,k,j}(x)}{1+x}}, & 
		            \text{if $j=0$ and $k \ge 1$} \\
    q_{n,k,j}(x), & \text{otherwise} 
		\end{cases}
\end{equation}
for $k \in \{0, 1,\dots,n+1\}$ and $j \in \{0, 
1\dots,n\}$.
\begin{proposition} \label{prop:qnmk} 
The following formulas hold for the $q_{n,k,j}(x)$
and $q^\ast_{n,k,j}(x)$:
\begin{itemize}
\itemsep=0pt
\item[{\rm (a)}]
$q^\ast_{n,k,0}(x) = q_{n,k-1}(x)$ for $k \in 
\{1, 2,\dots,n+1\}$.

\item[{\rm (b)}]
\[ q_{n,k}(x) = \sum_{j=0}^{n-1} q_{n-1,k,j}(x) 
              = \sum_{j=0}^{n-1} 
							  q^\ast_{n-1,k+1,j}(x) \]
for $k \in \{0, 1,\dots,n-1\}$.

\item[{\rm (c)}]
$q_{n,0,j}(x) = q^\ast_{n,0,j}(x) = 
q^\ast_{n,1,j}(x) = p_{n,j}(x)$ for $j \in 
\{0, 1,\dots,n\}$.

\item[{\rm (d)}]
$q_{n,k,k}(x) = q_{n,k+1,k}(x)$ for 
$k \in \{1, 2,\dots,n\}$.

\item[{\rm (e)}]
$q_{n,k,j}(x) = q_{n,k-1,j}(x) + 
xq_{n-1,k-1,j-1}(x)$ for $2 \le k \le j \le n$.

\item[{\rm (f)}]
For $k \ge 1$,
\begin{equation} \label{eq:q*nmk-1}
q^\ast_{n,k,j}(x) = x \sum_{i=0}^{j-1} 
                    q^\ast_{n-1,k-1,i}(x) + 
\sum_{i=j}^{n-1} q^\ast_{n-1,k-1,i}(x) 
\end{equation}
for $j \in \{1, 2,\dots,k-1\}$ and  
\begin{equation} \label{eq:q*nmk-2}
q^\ast_{n,k,j}(x) = x \sum_{i=0}^{j-1} 
                    q^\ast_{n-1,k,i}(x) + 
\sum_{i=j}^{n-1} q^\ast_{n-1,k,i}(x) 
\end{equation}
for $j \in \{k, k+1,\dots,n\}$.

\item[{\rm (g)}]
$q^\ast_{n,k,n}(x) = xq_{n,k-1}(x)$ for 
$k \in \{1, 2,\dots,n\}$.

\end{itemize}
\end{proposition}

\begin{proof}
Part (a) and the first equality of part (b) are 
immediate consequences of 
Equation~(\ref{eq:qnk-2}) and the defining 
equations~(\ref{eq:def-qnmk}) 
and~(\ref{eq:def-q*nmk}) of $q_{n,k,j}(x)$ and 
$q^\ast_{n,k,j}(x)$. For the second equality of 
part (b) we apply Equation~(\ref{eq:def-q*nmk}), 
the first equality of part (b), part (a) and the 
recurrence of Proposition~\ref{prop:qnk} (a), 
respectively, to get
\[ \sum_{j=0}^{n-1} q^\ast_{n-1,k+1,j}(x) = 
   \sum_{j=0}^{n-1} q_{n-1,k+1,j}(x) - x 
   q^\ast_{n-1,k+1,0} = q_{n,k+1}(x) - x q_{n-1,k}
	 (x) = q_{n,k}(x). \]

Similarly, part (c) follows from 
Equation~(\ref{eq:pnk-2}) and the definitions of 
$q_{n,k,j}(x)$ and $q^\ast_{n,k,j}(x)$ and part 
(d) follows from definition of $q_{n,k,j}(x)$, 
since $k+1$ cannot be a fixed point of a 
permutation $w \in \fS_{n+1}$ with $w^{-1}(1) = 
k+1$. To verify part (e), we split the right-hand
side of Equation~(\ref{eq:def-qnmk}) into 
two sums, running over permutations $w \in \fS_{n+1}$
with $w^{-1}(1) = j+1$ for which $k$ is a fixed 
point, or $k$ is not a fixed point, respectively. 
We then note that the first sum is equal to 
$(1+x) q_{n-1,k-1,j-1}(x)$ and the second to 
$q_{n,k-1,j}(x) - q_{n-1,k-1,j-1}(x)$.

Part (f) reduces to the 
recurrence~(\ref{eq:pnk-rec}) for the polynomials 
$p_{n,k}(x)$ for $k=1$ and follows from 
Equation~(\ref{eq:def-qnmk-alt}) by standard 
arguments for $k \ge 2$. For part (g) we denote 
by $\fix^\prime_k(w)$ the number of fixed points 
of $w$ in $\{2, 3,\dots,k\}$ and compute from 
Equation~(\ref{eq:def-qnmk}) that
\[ q^\ast_{n,k,n}(x) = x \sum_{w \in \fS_n} 
   (1+x)^{\fix^\prime_k(w)} x^{\exc(w)} = 
	 x \left( q_{n,k}(x) - xq_{n-1,k-1}(x) \right) 
	 = x q_{n,k-1}(x) \]
for $k \in \{1, 2,\dots,n\}$.
\end{proof}

\begin{proof}[Proof of Theorem~\ref{thm:main}] 
We claim that $(q_{n,k,j}(x))_{0 \le j \le n}$ 
is an interlacing sequence of real-rooted 
polynomials for all $n \in \NN$ and $k \in \{0, 
1,\dots,n+1\}$. 

Given the claim, 
Proposition~\ref{prop:qnmk} (b) implies that 
$q_{n,k}(x)$ is a real-rooted polynomial which 
is interlaced by $q^\ast_{n-1,k+1,0}(x) = 
q_{n-1,k}(x)$ for $k \in \{0, 1,\dots,n-1\}$. 
As a result, $q_{n,k}(x) \preceq 
x q_{n-1,k}(x)$ and hence $q_{n,k+1}(x) = 
q_{n,k}(x) + x q_{n-1,k}(x)$ is real-rooted and 
it is interlaced by $q_{n,k}(x)$ for every $k 
\in \{0, 1,\dots,n-1\}$. Since, as already 
mentioned, $q_{n,0}(x) = A_n(x)
\preceq \widetilde{A}_n(x) = q_{n,n}(x)$, it 
follows from Lemma~\ref{lem:interlace-trans} 
that $(q_{n,k}(x))_{0 \le k \le n}$ is an 
interlacing sequence of real-rooted polynomials 
for every $n \in \NN$. Therefore, by 
Lemma~\ref{lem:interlace-rec} (a), for every 
polynomial $p(x) = \sum_{k=0}^n c_k x^{n-k} 
(1+x)^k$ with $c_0, c_1,\dots,c_n \ge 0$,
\[ \iI_n (\aA^\circ (p(x))) = \sum_{k=0}^n c_k
   (q_{n,k}(x)) \]
is a real-rooted polynomial such that 
\[ A_n(x) = q_{n,0}(x) \preceq \iI_n (\aA^\circ 
   (p(x))) \preceq q_{n,n}(x) = \widetilde{A}_n(x). 
	 \] 
Since $A_n(x)$ and $\widetilde{A}_n(x)$ are 
symmetric, with centers of symmetry $(n-1)/2$ and 
$n/2$, respectively, this implies that 
$\widetilde{A}_n(x) \preceq \aA^\circ (p(x)) 
\preceq xA_n(x)$ and proves part (a). Given that 
$A_n(x) \preceq xA_n(x)$, the interlacing 
relations 
\[ A_n(x) \preceq \iI_n (\aA^\circ(p(x))) 
   \preceq \widetilde{A}_n(x) \preceq \aA^\circ 
	 (p(x)) \preceq xA_n(x) \]
and Lemma~\ref{lem:interlace-trans} imply that 
$\iI_n (\aA^\circ(p(x))) \preceq \aA^\circ(p(x))$.
Since $\aA^\circ(p(x))$ is already known to have
a nonnegative (even unimodal) symmetric 
decomposition with respect to $n$ 
\cite[Theorem~3.2]{BJ22}, the latter interlacing
relation and \cite[Theorem~2.7]{BS21} imply part
(b). 

We now prove the claim (the proof reduces to part 
of the proof of \cite[Theorem~4.4]{BJ22} in the
special case $k=n+1$). Proceeding by 
induction on $n$ and $k$, we assume that $n \ge 1$
and $k \in \{0, 1,\dots,n+1\}$ and that the result 
holds for all pairs which are lexicographically 
smaller than $(n,k)$. By Proposition~\ref{prop:qnmk} 
(c), for $k \in \{0, 1\}$ the claim reduces to the 
fact that $(p_{n,j}(x))_{0 \le j \le n}$ is an 
interlacing sequence of real-rooted polynomials for 
every $n \in \NN$ (see Example~\ref{ex:pnk}). 
Suppose that $k \ge 2$. By parts (a) and (b) of 
Proposition~\ref{prop:qnmk} we have
\[ q^\ast_{n,k,0}(x) = q_{n,k-1}(x) = 
           \sum_{j=0}^{n-1} q_{n-1,k-1,j}(x) 
= (1+x) q^\ast_{n-1,k-1,0}(x) + 
  \sum_{j=1}^{n-1} q^\ast_{n-1,k-1,j}(x). \]
Applying Lemma~\ref{lem:interlace-rec} (b) and 
the induction hypothesis to this expression, 
combined with those of 
Proposition~\ref{prop:qnmk} (f), we conclude 
that $(q^\ast_{n,k,j}(x))_{0 \le j \le k-1}$ 
and $(q^\ast_{n,k,j}(x))_{k \le j \le n}$ are 
interlacing sequences of real-rooted polynomials. 
Since the second sequence is empty for $k=n+1$, 
the claim holds in this case. Thus, we may 
assume that $2 \le k \le n$. By 
Lemma~\ref{lem:interlace-trans}, it suffices to 
prove that $q^\ast_{n,k,k-1}(x) \preceq 
q^\ast_{n,k,k}(x)$ and $q^\ast_{n,k,0}(x) 
\preceq q^\ast_{n,k,n}(x)$. The latter holds 
because $q^\ast_{n,k,n}(x) = x q^\ast_{n,k,0}
(x)$ by parts (a) and (g) of 
Proposition~\ref{prop:qnmk}. For the former, 
recall from Proposition~\ref{prop:qnmk} (e) 
that 
\[ q^\ast_{n,k,k}(x) = q^\ast_{n,k-1,k}(x) + x 
                   q^\ast_{n-1,k-1,k-1}(x). \]
By Proposition~\ref{prop:qnmk} (d) and our 
induction hypothesis, $q^\ast_{n,k,k-1}(x) = 
q^\ast_{n,k-1,k-1}(x) \preceq q^\ast_{n,k-1,k}(x)$. 
Moreover, since 
$(q_{n-1,k-1,j}(x))_{0 \le j \le n}$ is interlacing 
by induction, each term on the right-hand side of 
Equation~(\ref{eq:q*nmk-1}) for $j=k-1$ interlaces 
$xq^\ast_{n-1,k-1,k-1}(x)$ and hence so does their  
sum $q^\ast_{n,k,k-1}(x)$. As a result, 
\[ q^\ast_{n,k,k-1}(x) \preceq q^\ast_{n,k-1,k}(x) 
   + x q^\ast_{n-1,k-1,k-1}(x) = q^\ast_{n,k,k}(x) 
	\]
and the proof follows. 
\end{proof}

\begin{example} \label{ex:simple-cor} \rm 
Applying Theorem~\ref{thm:main} to $p(x) = c_0 x^n 
+ c_1 x^{n-1}(1+x) + c_2 x^{n-2}(1+x)^2$ shows that 
$a A_n(x) + b A_{n-1}(x) + c A_{n-2}(x)$ is 
real-rooted whenever $a \ge b-c \ge c \ge 0$.
Moreover, it has a real-rooted and interlacing 
symmetric decomposition with respect to $n$.
\qed
\end{example}

\section{Combinatorics of the derangement 
         transformation}
\label{sec:derange}

The \emph{derangement transformation} was 
introduced and studied by Br\"and\'en and 
Solus~\cite[Section~3.2]{BS21}; it is the linear 
map $\dD: \RR[x] \to \RR[x]$ defined by setting 
$\dD(x^n) = d_n(x)$ for $n \in \NN$, where 
\[ d_n(x) = \sum_{i=0}^n \, (-1)^i {n \choose i} 
   A_{n-i}(x) = \sum_{w \in \fS_n : \, \Fix(w) 
	 = \varnothing} x^{\exc(w)} \]
is the \emph{$n$th derangement polynomial} and 
$\Fix(w)$ is the set of fixed points of a 
permutation $w \in \fS_n$. This polynomial is 
known to be $\gamma$-positive (in particular, 
symmetric), with center of symmetry $n/2$, and 
real-rooted; see \cite[Section~2.1.4]{Ath18} 
\cite[Section~3.2]{BS21} \cite[Section~3.1]{HZ19}
and references therein. 

As shown in~\cite[Corollary~3.7]{BS21}, the 
derangement transformation satisfies the 
analogue of part (a) of Theorem~\ref{thm:main}.  
This section shows that the reciprocals of the 
polynomials 
\begin{equation}  
d_{n,k}(x) = \dD \left( x^k(1+x)^{n-k} \right)  
\label{eq:def-dnk}
\end{equation} 
have $\gamma$-positive (in fact, real-rooted 
and interlacing) symmetric decompositions with 
respect to $n$, for $k \in \{0, 1,\dots,n\}$, 
a property that will be one of the ingredients 
of the proof of 
Theorem~\ref{thm:gamma}, and confirms the 
analogue of part (b) of Theorem~\ref{thm:main} 
for $\dD$ (see Corollary~\ref{thm:derange}). 
The polynomials $d_{n,k}(x)$ are shown on 
Table~\ref{tab:dnk} for $n \le 4$. 

{\scriptsize
\begin{table}[hptb]
\begin{center}
\begin{tabular}{| l || l | l | l | l | l | l ||} \hline
& $k=0$ & $k=1$ & $k=2$ & $k=3$ & $k=4$ \\ \hline \hline
$n=0$   & 1 &  & & & \\ \hline
$n=1$   & 1 & 0 & & & \\ \hline
$n=2$  & $1+x$ & $x$ & $x$ & & \\ \hline
$n=3$  & $1+4x+x^2$ & $3x+x^2$ & $2x+x^2$ & 
         $x+x^2$ & \\ \hline
$n=4$  & $1+11x+11x^2+x^3$ & $7x+10x^2+x^3$ & 
         $4x+9x^2+x^3$ & $2x+8x^2+x^3$ & $x+7x^2+x^3$
				 \\ \hline
\end{tabular}

\bigskip
\caption{The polynomials $d_{n,k}(x)$ for $n \le 4$.}
\label{tab:dnk}
\end{center}
\end{table}}

Part (b) of the following statement, which may be  
of independent interest, shows that $d_{n,k}(x)$ 
can be written as a sum of two $\gamma$-positive 
polynomials with centers of symmetry $n/2$ and 
$(n-1)/2$. For $k \in \{0, n\}$, it reduces to 
known expressions (see, for instance, 
\cite[Theorems~2.1 and~2.13]{Ath18}) for $A_n(x)
= d_{n,0}(x)$ and $d_n(x) = d_{n,n}(x)$ which 
demonstrate their $\gamma$-positivity. The proof 
extends the proof of $\gamma$-positivity of 
$d_n(x)$, given in \cite[Section~4]{AS12}. We 
recall that a \emph{decreasing run} of a 
permutation $w \in \fS_n$ is defined as a 
maximal set of integers of the form 
$\{a, a+1,\dots,b\} \subseteq [n]$ such that 
$w(a) > w(a+1) > \cdots > w(b)$.

\begin{proposition} \label{prop:dnk} 
For every positive integer $n$ and every $k \in \{0, 
1,\dots,n\}$:

\begin{itemize}
\itemsep=0pt
\item[{\rm (a)}]
\[ d_{n,k}(x) = \sum_{i=0}^k (-1)^i {k \choose i} 
A_{n-i}(x) = \sum_{w \in \fS_n : \, \Fix(w) 
   \subseteq [n-k]} x^{\exc(w)}, \]

\item[{\rm (b)}]
\[ d_{n,k}(x) = \sum_{i=0}^{\lfloor n/2 \rfloor} 
                \xi^+_{n,k,i} \, x^i (1+x)^{n-2i} 
                \ + 
                \sum_{i=0}^{\lfloor (n-1)/2 \rfloor} 
							\xi^-_{n,k,i} \, x^i (1+x)^{n-1-2i}, 
\]
where 

\medskip
\begin{itemize}
\itemsep=0pt
\item[$\bullet$] 
$\xi^+_{n,k,i}$ is equal to the number of permutations 
$w \in \fS_n$ with $w(1) > n-k$ which have $i$ 
decreasing runs and none of size one,
 
\item[$\bullet$] 
$\xi^-_{n,k,i}$ is equal to the number of permutations 
$w \in \fS_n$ with $w(1) \le n-k$ which have $i$ 
decreasing runs and none, except possibly the 
first, of size one.
\end{itemize}

\item[{\rm (c)}]
The polynomial $\iI_n (d_{n,k}(x))$ has a 
$\gamma$-positive symmetric decomposition with 
respect to $n$.
\end{itemize}
\end{proposition}

\begin{proof}
The first formula of part (a) follows from 
Equation~(\ref{eq:def-dnk}) by expanding 
$x^k(1+x)^{n-k}$ as $x^k(1+x)^{n-k} = 
(1+x-1)^k(1+x)^{n-k} = \sum_{i=0}^k (-1)^i 
{k \choose i} (1+x)^{n-i}$, applying the linearity 
of $\dD$ and recalling that 
\[ \dD \left( (1+x)^n \right) = \sum_{i=0}^n  
{n \choose i} d_i(x) = A_n(x) \]
for $n \in \NN$. The second formula follows from 
the first by a standard inclusion-exclusion 
argument.

For part (b) we extend the proof of the special 
case $k=n$, given in \cite[Section~4]{AS12}. To 
sketch this argument, which is based on the idea 
of valley hopping \cite{FS74, FS76} (see also 
\cite[Section~4.1]{Ath18} 
\cite[Section~4.2]{Pet15}), 
we recall some notation and terminology. 
Let $w = (w_1, w_2,\dots,w_n) \in \fS_n$ be a 
permutation, written in one-line notation, so 
that $w_i = w(i)$ for $i \in [n]$. An 
\emph{ascent} (respectively, \emph{descent}) of
$w$ is any index $i \in [n-1]$ such that $w_i < 
w_{i+1}$ (respectively, $w_i > w_{i+1}$). The 
number of ascents (respectively, descents) of
$w$ is denoted by $\asc(w)$ (respectively, 
$\des(w)$). An entry $w_j$ is said to be a 
\emph{left-to-right maximum} of $w$ if $w_i < 
w_j$ for all $i<j$. We will denote by 
$\eE_{n,k}$ the set of all permutations $w \in 
\fS_n$ such that $j$ is a descent of $w$ for 
every left-to-right maximum $w_j > k$ of $w$. 

The standard representation (essentially, the 
fundamental transformation, mentioned in 
Section~\ref{sec:euler}) of permutations, as 
described in \cite[p.~23]{StaEC1} and 
\cite[Section~4]{AS12}, and the second 
expression of Proposition~\ref{prop:dnk} (a) 
for $d_{n,k}(x)$ show that 
\[ d_{n,k}(x) = \sum_{w \in \eE_{n,n-k}} 
   x^{\asc(w)}. \]
We claim that 
\begin{align} \label{eq:dnk+}
\sum_{w \in \eE_{n,n-k} : \, w(1) > n-k} 
x^{\asc(w)} & = \sum_{i=0}^{\lfloor n/2 \rfloor} 
                \xi^+_{n,k,i} \, x^i (1+x)^{n-2i} 
\\ \label{eq:dnk-}
\sum_{w \in \eE_{n,n-k} : \, w(1) \le n-k} 
x^{\asc(w)} & = \sum_{i=0}^{\lfloor (n-1)/2 \rfloor} 
								\xi^-_{n,k,i} \, x^i (1+x)^{n-1-2i}, 
\end{align}
where $\xi^+_{n,k,i}$ and $\xi^-_{n,k,i}$ are as in 
the statement of the proposition. This implies part 
(b). 

Let $w = (w_1, w_2,\dots,w_n) \in 
\eE_{n,n-k}$ and let us assume first that $w_1 > n-k$ 
or, equivalently, that all left-to-right maxima of 
$w$ are larger that $n-k$. As a result, all
left-to-right maxima of $w$ are located at descents
of $w$ and, in particular, $w_1 > w_2$. 
We set $w_0 = 0$ and $w_{n+1} = n+1$ and call $i \in 
[n]$ a \emph{double ascent} (respectively, 
\emph{double descent}) of $w$ if $w_{i-1} < w_i < 
w_{i+1}$ (respectively, $w_{i-1} > w_i > w_{i+1}$). 
Given a double ascent or double descent $i$ of $w$, 
and setting $w_i = a$, we 
define the permutation $\varphi_a (w) \in \fS_n$ as 
follows: if $i$ is a double ascent of $w$, then 
$\varphi_a (w)$ is the permutation obtained from $w$ 
by moving $w_i = a$ between $w_j$ and $w_{j+1}$, 
where $j$ is the largest index satisfying $1 \le j 
< i$ and $w_j > w_i > w_{j+1}$ (such an index 
exists because $i$ is not a descent of $w$ 
and hence $a$ is not a left-to-right maximum). 
Similarly, if $i$ is a double descent of $w$, then 
$\varphi_a (w)$ is the permutation obtained from $w$ 
by moving $w_i = a$ between $w_j$ and $w_{j+1}$, 
where $j$ is the smallest index satisfying $i < j 
\le n$ and $w_j < w_i < w_{j+1}$ (such an index 
exists because $w_{n+1} = n+1$). We set $\varphi_a(w) 
= w$ for all other $a \in [n]$. 

We observe that $w$ and
$\varphi_a(w)$ have the same left-to-right maxima (all 
larger that $n-k$) for all $w \in \eE_{n,n-k}$ and $a 
\in [n]$ and conclude that the maps $\varphi_a$ are 
pairwise commuting involutions on the set of 
elements of $\eE_{n,n-k}$ with first entry larger 
than $n-k$. Thus, they define a $\ZZ^n_2$-action on 
this set. Moreover, each orbit of this action has 
a unique element $u$ having no double ascent 
(equivalently, all decreasing runs of $u$ have size 
at least two). As in the proof of 
\cite[Theorem~1.4]{AS12} we find that 
\[ \sum_{w \in \Orb(u)} x^{\asc(w)} = x^{\asc(u)}
(1+x)^{n-2\asc(u)}, \]
where $\Orb(u)$ stands for the orbit of $u$. Summing 
over all orbits yields Equation~(\ref{eq:dnk+}).   

Finally, we consider permutations $w = 
(w_1, w_2,\dots,w_n) \in \eE_{n,n-k}$ such that $w_1 
\le n-k$. We now set $w_0 = w_{n+1} = n+1$ and define
$\varphi_a(w)$ for $a \in [n]$ using the same rules as 
before. We leave to the reader to verify that 
$\varphi_a(w)$ has the same left-to-right maxima larger 
that $n-k$ as $w$, but those which are less than $n-k$ 
may differ; in particular, $w$ and $\varphi_a(w)$ may 
have different first entries. Nevertheless, the first 
entry of $\varphi_a(w)$ is still no larger than $n-k$ 
and the maps $\varphi_a$ define a $\ZZ^n_2$-action on 
the set of elements of $\eE_{n,n-k}$ with first entry 
no larger than $n-k$. Again, each orbit of this action 
has a unique element $u = (u_1, u_2,\dots,u_n)$ having 
no double ascent, although it is now possible that 
$u_1 < u_2$ (equivalently, all decreasing runs of $u$ 
other than the first have size at least two). As in 
the special case $k=0$ (see, for instance, 
\cite[Section~4.1]{Ath18} 
\cite[Section~4.2]{Pet15}) we find that 
\[ \sum_{w \in \Orb(u)} x^{\asc(w)} = x^{\asc(u)}
(1+x)^{n-1-2\asc(u)}. \]
Summing over all orbits yields 
Equation~(\ref{eq:dnk-}). Part (c) follows from 
part (b).
\end{proof}

As a consequence of~\cite[Theorem~3.6]{BS21}, 
each polynomial $d_{n,k}(x)$ is real-rooted and 
the sequence $(d_{n,k}(x))_{0 \le k \le n}$ is 
interlacing for every $n \in \NN$. The 
following statement implies that the symmetric 
decompositions of Proposition~\ref{prop:dnk} 
(c) are, in fact, real-rooted and interlacing.

\begin{corollary} \label{thm:derange} 
The polynomial $\iI_n (\dD(p(x)))$ has a 
real-rooted and interlacing symmetric 
decomposition with respect to $n$ for every 
$p(x) \in \pP_n[x]$.
\end{corollary}

\begin{proof}
As already mentioned, it has been shown 
in~\cite[Corollary~3.7]{BS21} that $\dD(p(x))$
is real-rooted and that it is interlaces 
$\iI_n (\dD(p(x)))$. Thus, 
by~\cite[Theorem~2.7]{BS21}, it suffices to 
confirm that $\iI_n (\dD(p(x)))$ has a 
nonnegative symmetric decomposition with respect 
to $n$. Since, 
\[ \iI_n (\dD (p(x))) = \sum_{k=0}^n c_k \, 
   \iI_n (d_{n,k}(x)) \]
for every $p(x) = \sum_{k=0}^n c_k x^k 
(1+x)^{n-k} \in \pP_n[x]$, this is guaranteed 
by Proposition~\ref{prop:dnk} (c).
\end{proof}

\section{Simplicial complexes and their 
         triangulations}
\label{sec:geom}

This section includes background on simplicial 
complexes, their triangulations and their face 
enumeration which is essential to understand and
prove Theorem~\ref{thm:gamma}. Moreover, the 
polynomials $q_{n,k}(x)$ and $d_{n,k}(x)$ of 
Sections~\ref{sec:euler} and~\ref{sec:derange} 
are generalized in the setting of uniform 
triangulations and basic properties of these 
generalizations are discussed.

\subsection{Simplicial complexes}
\label{sec:complexes}

We assume familiarity with basic notions, such as 
abstract and geometric simplicial complexes and 
the correspondence between them, as explained, for 
instance, in \cite{Bj95, HiAC, StaCCA}. All 
simplicial complexes considered here will be 
abstract and finite. Following~\cite{Ath22}, we 
denote by $\sigma_n$ the abstract simplex $2^V$ 
on an $n$-element vertex set $V$. 

A fundamental enumerative invariant of a 
simplicial complex $\Delta$ is the 
\emph{$h$-polynomial}, defined by the formula
\begin{equation}
\label{eq:hdef}
h(\Delta, x) = \sum_{i=0}^n f_{i-1} (\Delta) \, 
x^i (1-x)^{n-i}, 
\end{equation}
where $f_i(\Delta)$ is the number of 
$i$-dimensional faces of $\Delta$ and $n-1$ is 
its dimension. The $h$-polynomial has nonnegative 
coefficients if $\Delta$ triangulates a ball or a 
sphere \cite[Chapter~II]{StaCCA}. Moreover, it is
symmetric in the latter case, with center of 
symmetry $n/2$.  

Suppose that $\Delta$ triangulates the 
$(n-1)$-dimensional ball. The \emph{interior 
$h$-polynomial} $h^\circ(\Delta,x)$ of $\Delta$ 
is defined by the right-hand side of 
Equation~(\ref{eq:hdef}) when $f_{i-1}(\Delta)$ 
is replaced by the number of $(i-1)$-dimensional 
interior faces of $\Delta$. Then, 
$h^\circ(\Delta,x) = x^n h(\Delta,1/x)$ 
(see, for instance, \cite[Lemma~6.2]{Sta87}), a 
formula which the reader may wish to consider as 
the definition of $h^\circ(\Delta,x)$. The 
\emph{theta polynomial} of $\Delta$ is defined 
as
\begin{equation}
\label{eq:theta-def}
\theta(\Delta, x) = h(\Delta,x) - h(\partial 
                      \Delta, x),
\end{equation}
where $\partial \Delta$ is the boundary complex 
of $\Delta$. This polynomial is symmetric, with 
center of symmetry $n/2$, and under some mild 
hypotheses it has nonnegative coefficients. For 
other basic properties of theta polynomials we 
refer the reader to~\cite{Ath22+}, where their 
role in the enumerative theory of triangulations 
of simplicial complexes is also explained.

\subsection{Triangulations}
\label{sec:triang}

By the term \emph{triangulation} of a simplicial 
complex $\Delta$ we 
will always mean a geometric triangulation. Thus, 
a simplicial complex $\Delta'$ is a triangulation 
of $\Delta$ if there exists a geometric realization 
$K'$ of $\Delta'$ which geometrically subdivides 
a geometric realization $K$ of $\Delta$. The 
restriction of $\Delta'$ to a face $F \in \Delta$ 
is a triangulation of the simplex $2^F$ denoted by 
$\Delta'_F$.

The \emph{local $h$-polynomial} of a triangulation 
$\Gamma$ of a simplex $2^V$ is defined 
\cite[Definition~2.1]{Sta92} by the formula 
\begin{equation} \label{eq:def-local-h}
  \ell_V (\Gamma, x) \ = \sum_{F \subseteq V} 
  \, (-1)^{|V \sm F|} \, h (\Gamma_F, x).
\end{equation}
Stanley~\cite{Sta92} showed 
that $\ell_V (\Gamma, x)$ has nonnegative
coefficients and that it is symmetric, with center
of symmetry $|V|/2$. The significance of local 
$h$-polynomials stems from Stanley's Locality 
Formula~\cite[Theorem~3.2]{Sta92}, which expresses
the $h$-polynomial of a triangulation $\Delta'$ of
a pure simplicial complex $\Delta$ in terms of the
local $h$-polynomials of the restrictions of 
$\Delta'$ to the faces of $\Delta$ and the 
$h$-polynomials of the links of these faces in 
$\Delta$.

\medskip
\noindent
\textbf{Barycentric subdivision.}
Barycentric subdivision is a prototypical example 
of uniform triangulation of a simplicial complex. 
The \emph{barycentric subdivision} of a 
simplicial complex $\Delta$, denoted by 
$\sd(\Delta)$, is defined as the simplicial complex 
which consists of all chains of nonempty faces of 
$\Delta$. As is well known, $\sd(\Delta)$ can be 
realized as a triangulation of $\Delta$. 

\medskip
\noindent
\textbf{The antiprism construction.}
Every triangulation $\Gamma$ of a simplex $2^V$ can 
be extended to a triangulation of a sphere of the 
same dimension as follows. Let $V = \{v_1, 
v_2,\dots,v_n\}$ and pick an $n$-element set $U = 
\{u_1, u_2,\dots,u_n\}$ which is disjoint from the 
vertex set of $\Gamma$. The \emph{antiprism sphere} 
over $\Gamma$ \cite[Section~4]{Ath12},
denoted by $\Delta_\aA(\Gamma)$, is defined as the 
collection of sets of the form $E \cup G$, where 
$E = \{ u_i: i \in I\}$ is a face of the simplex 
$2^U$ for some $I \subseteq [n]$ and $G$ is a face 
of the restriction $\Gamma_F$ of $\Gamma$ to the 
face $F = \{ v_j: j \in [n] \sm I\}$ of the simplex 
$2^V$ which is complementary to $E$. The complex 
$\Delta_\aA(\Gamma)$ is a triangulation of the 
$(n-1)$-dimensional sphere which contains $2^U$ 
and $\Gamma$ as subcomplexes and naturally 
triangulates $\Delta_\aA(2^V)$; for other basic 
properties, see \cite[Proposition~4.6]{Ath12} 
\cite[Proposition~4.1]{Ath20} and 
\cite[Section~3]{ABK22}.

\subsection{Uniform triangulations}
\label{subsec:uniform}

Let $\fF = (f_\fF(i,j))_{0 \le i \le j \le n}$ 
be a triangular array of nonnegative integers. 
A triangulation $\Gamma$ of the simplex $\sigma_n
= 2^V$ is called \emph{$\fF$-uniform} \cite{Ath22} 
if for all $0 \le i \le j \le n$, the restriction 
of $\Gamma$ 
to any face of $\sigma_n$ of dimension $j-1$ has 
exactly $f_\fF(i,j)$ faces of dimension $i-1$, and 
\emph{uniform} if it is \emph{$\fF$-uniform} 
for some $\fF$ (these definitions extend naturally
to triangulations of any simplicial complex of 
dimension less than $n$). The array 
$\fF$ is called the \emph{$f$-triangle} associated 
to $\Gamma$. 

We define the maps $\hH^\circ_\fF, \lL_\fF: \RR_n[x] 
\to \RR_n[x]$ by setting
\begin{align} 
\hH^\circ_\fF (x^m) & = h^\circ_\fF (\sigma_m,x) 
                      = x^m h_\fF (\sigma_m,1/x) 
\label{eq:def-Hcirc} \\
\lL_\fF (x^m) & = \ell_\fF(\sigma_m,x) 
     \label{eq:def-L}
\end{align}
for $m \in \{0, 1,\dots,n\}$ and extending by 
linearity, where  $h_\fF (\sigma_m,x)$, $h^\circ_\fF 
(\sigma_m,x)$ and $\ell_\fF(\sigma_m,x)$ are the 
$h$-polynomial, the interior $h$-polynomial and the 
local $h$-polynomial, respectively, of the 
restriction of $\Gamma$ to any $(m-1)$-dimensional 
face of $\sigma_n$. In analogy with the cases of 
Eulerian and derangement transformations, we 
consider the polynomials $q_{\fF,m,k}(x)$ and 
$\ell_{\fF,m,k}(x)$ defined by 
\begin{align} 
\iI_m (q_{\fF,m,k}(x)) & = \hH_\fF^\circ 
             \left( x^{m-k}(1+x)^k \right)  
       \label{eq:def-qFnk} \\
\ell_{\fF,m,k}(x)) & = \lL_\fF 
             \left( x^k(1+x)^{m-k} \right)
\end{align}
for all integers $0 \le k \le m \le n$. We 
have $q_{\fF,m,0}(x) = \ell_{\fF,m,0}(x) = 
h_\fF(\sigma_m,x)$ and $\ell_{\fF,m,m}(x) = 
\ell_\fF(\sigma_m,x)$ for every $m \in 
\{0, 1,\dots,n\}$.

\begin{example} \label{ex:qFnk} \rm 
Let $\Gamma$ be the barycentric subdivision
of $2^V$. Then, 
\[ \hH_\fF^\circ (x^m) = x^m h_\fF(\sigma_m,1/x) 
   = x^m A_m(1/x) = 
	 \begin{cases} 1, & \text{if $m=0$} \\
    xA_m(x), & \text{if $m \ge 1$} \end{cases} \]
and 
\[ \lL_\fF (x^m) = \ell_\fF(\sigma_m,x) = d_m(x)  
   \]
for $m \in \{0, 1,\dots,n\}$; see, for instance,
\cite[Section~3.3.1]{Ath18} 
\cite[Proposition~2.4]{Sta92}. Hence, 
$\hH_\fF^\circ$ and $\lL_\fF$ are the restrictions 
of $\aA^\circ$ and $\dD$ to $\RR_n[x]$ and 
$q_{\fF,m,k}(x) = q_{m,k}(x)$, $\ell_{\fF,m,k}(x) 
= d_{m,k}(x)$ for all $m, k$.
\qed
\end{example}

\begin{proposition} \label{prop:geom} 
Let $\Gamma$ be a uniform triangulation of an 
$(n-1)$-dimensional simplex $2^V$ with associated 
$f$-triangle $\fF$.  
\begin{itemize}
\itemsep=0pt
\item[{\rm (a)}]
The polynomials $q_{\fF,m,k}(x)$ and $\ell_{\fF,m,k}
(x)$ satisfy the recurrence 
\begin{align*} 
  q_{\fF,m,k+1}(x) & = q_{\fF,m,k}(x) + 
  xq_{\fF,m-1,k}(x) \\ 
	\ell_{\fF,m,k+1}(x) & = \ell_{\fF,m,k}(x) - 
  \ell_{\fF,m-1,k}(x)
\end{align*}
for $0 \le k < m \le n$.

\item[{\rm (b)}]
We have
\begin{align*} 
q_{\fF,m,k}(x) & = \sum_{i=0}^k {k \choose i} x^i 
                 h_\fF(\sigma_{m-i},x) \\
\ell_{\fF,m,k}(x) & = \sum_{i=0}^k (-1)^i  
{k \choose i} h_\fF(\sigma_{m-i},x) 							
\end{align*}						
for $0 \le k \le m \le n$.

\item[{\rm (c)}]
Let $\Delta_\aA(\Gamma)$ be the antiprism sphere 
over $\Gamma$, with new vertices $u_1, 
u_2,\dots,u_n$. Then, for every $k \in 
\{0, 1,\dots,n\}$, $q_{\fF,n,k}(x)$ is equal to 
the $h$-polynomial of the induced subcomplex 
of $\Delta_\aA(\Gamma)$ obtained by removing 
vertices $u_{k+1},\dots,u_n$.

\item[{\rm (d)}]
The polynomial $q_{\fF,n,n}(x)$ is symmetric, with 
center of symmetry $n/2$. 
\end{itemize}
\end{proposition}

\begin{proof}
The proofs of (a) and (b) are identical to those of 
the corresponding statements of 
Proposition~\ref{prop:qnk} for the $q_{\fF,m,k}(x)$
and follow by similar reasoning for the 
$\ell_{\fF,m,k}(x)$. Given part (b), one can extend 
the proof of part (d) of \cite[Proposition~4.1]{Ath20} 
(which corresponds to the special case $k=n$) to 
prove part (c). To sketch a more direct argument, 
let $\Delta$ be the induced subcomplex of 
$\Delta_\aA(\Gamma)$ obtained by removing vertices 
$u_{k+1},\dots,u_n$. By the defining 
equation~(\ref{eq:hdef}) of the $h$-polynomial we 
have
\[ h(\Delta,x) = \sum_{F \in \Delta} x^{|F|} 
    (1-x)^{n-|F|} = \sum_{I \subseteq [k]} 
		\sum_{F \in \Delta_I} x^{|F|} (1-x)^{n-|F|}, \]
where $\Delta_I$ consists of those faces $F \in 
\Delta$ for which $F \cap U = \{u_i: i \in I\}$.
One then recognizes the inner sum as $x^{|I|}$ times
the $h$-polynomial of the restriction of $\Gamma$ to
the face $\{ v_j: j \in [n] \sm I\}$ of $2^V$ which 
is complementary to $\{u_i: i \in I\}$ and concludes 
that 
\[ h(\Delta,x) = \sum_{\{v_{k+1},\dots,v_n\} \subseteq 
   F \subseteq V} x^{n-|F|} h(\Gamma_F,x) = \sum_{i=0}^k 
	 {k \choose i} x^i h_\fF(\sigma_{n-i},x). \]
By part (b), the latter sum is equal to $q_{\fF,n,k}
(x)$. Part (d) follows from part (c), which implies 
that $q_{\fF,n,n}(x) = h (\Delta_\aA(\Gamma),x)$ is 
the $h$-polynomial of a triangulation of the 
$(n-1)$-dimensional sphere.
\end{proof}

\section{Gamma-positivity of the $\hH_\fF^\circ$ and
         $\lL_\fF$ transformations}
\label{sec:uniform}

This section proves Theorem~\ref{thm:gamma}. We first 
introduce a family of enumerative invariants of a 
triangulation $\Gamma$ of a simplex $2^V$ which 
provides a common generalization of the 
$h$-polynomial and the local $h$-polynomial of 
$\Gamma$. This generalization is different from the 
relative local $h$-polynomial $\ell_V(\Gamma, E, x)$, 
introduced in \cite[Section~3]{Ath12}.

\begin{definition} \label{def:local-VE} 
Given a triangulation $\Gamma$ of the simplex $2^V$ 
and a face $E \subseteq V$ of $2^V$, the local 
$h$-polynomial of $\Gamma$ with respect to $V$ and 
$E$ is defined by the formula 
\begin{equation} \label{eq:local-VE-def}
\ell_{V,E} (\Gamma, x) =
\sum_{E \subseteq F \subseteq V} (-1)^{|V \sm F|}
h(\Gamma_F, x). 
\end{equation} 
\end{definition}

\begin{remark} \label{rem:local-VE} \rm 
(a) By definition, we have $\ell_{V,\varnothing} 
(\Gamma, x) = \ell_V (\Gamma, x)$ and $\ell_{V,V} 
(\Gamma, x) = h(\Gamma, x)$. We also have $\ell_{V,E} 
(\Gamma, x) = h(\Gamma, x) - h(\Gamma_E,x)$ for every
facet (maximal with respect to inclusion face) $E$ of 
$\partial(2^V)$.

(b) By inclusion-exclusion \cite[Section~2.1]{StaEC1},
for given $E \subseteq G \subseteq V$,
\[ h(\Gamma_G, x) = \sum_{E \subseteq F 
   \subseteq G} \ell_{F,E} (\Gamma_F, x). \]

(c) Combining the defining 
equation~(\ref{eq:local-VE-def}) of $\ell_{V,E}
(\Gamma, x)$ with the identity of part (b), applied 
for $E = \varnothing$, we get
\begin{align*}
\ell_{V,E} (\Gamma, x) & = 
\sum_{E \subseteq G \subseteq V} (-1)^{|V \sm G|}
h(\Gamma_G, x) = 
\sum_{E \subseteq G \subseteq V} (-1)^{|V \sm G|} 
\left( \, \sum_{F \subseteq G} \ell_F(\Gamma_F, x) 
\right) \\ & =  
\sum_{F \subseteq V} \ell_F(\Gamma_F, x) 
\left( \, \sum_{E \cup F \subseteq G \subseteq V} 
(-1)^{|V \sm G|} \right) = \sum_{V \sm E \subseteq 
F \subseteq V} \ell_F (\Gamma_F, x)
\end{align*}
for every $E \subseteq V$. This formula shows 
that $\ell_{V,E} (\Gamma, x)$ has nonnegative 
coefficients.
\qed
\end{remark}

\begin{example} \label{ex:local-VE} \rm 
(a) By part (a) of Proposition~\ref{prop:dnk}, in 
the special case of barycentric subdivision we have 
$\ell_{V,E} (\Gamma, x) = d_{|V|,|V \sm E|}(x)$ for
every $E \subseteq V$. 

(b) More generally, let $\Gamma$ be the $r$-fold 
edgewise subdivision of the barycentric subdivision 
(termed 
in \cite[Section~3]{Ath22} as the \emph{$r$-colored 
barycentric subdivision}) of $2^V$. Let $n=|V|$ and 
$\ZZ_r \wr \fS_n$ be the group of $r$-colored 
permutations of the set $[n]$ (see 
\cite[Section~5]{Ath20} for any undefined terminology 
about $r$-colored permutations). By 
\cite[Proposition~5.1]{Ath20} and an application of
the inclusion-exclusion principle we have
\begin{equation} \label{eq:fexc}
  \ell_{V,E} (\Gamma, x) = 
  \sum_{w \in (\ZZ_r \wr \fS_n)^b : \, \Fix(w) 
	\subseteq [k]} x^{\fexc(w)/r}, 
\end{equation}
where $\fexc(w)$ is the flag excedance number of $w 
\in \ZZ_r \wr \fS_n$, $\Fix(w)$ is the set of fixed 
points of $w \in \ZZ_r \wr \fS_n$ of zero color, 
$(\ZZ_r \wr \fS_n)^b$ is the set of $w \in \ZZ_r \wr 
\fS_n$ such that the sum of the colors of the 
coordinates of $w$ is divisible by $r$ and $k = |E|$.
This polynomial reduces to $d_{n,n-k}(x)$ for $r=1$.
\qed
\end{example}

The following two formulas were exploited 
in~\cite{Ath22+} in order to investigate unimodality
and $\gamma$-positivity properties of $h$-polynomials
of triangulations of balls and local $h$-polynomials 
of triangulations of simplices:
\begin{align} 
h(\Gamma, x) & = \sum_{F \subseteq V} 
\theta (\Gamma_F, x) \, A_{|V \sm F|} (x)
\label{eq:Ath22} \\
\ell_V (\Gamma, x) & = \sum_{F \subseteq V}
\theta (\Gamma_F, x) \, d_{|V \sm F|} (x).
\label{eq:KMS}
\end{align} 

The second formula is equivalent to 
\cite[Theorem~4.7]{KMS19} and the first follows from 
that and the formula $h(\Gamma, x) = 
\sum_{F \subseteq V} \ell_F(\Gamma_F, x)$; see 
\cite[Theorem~3.4]{Ath22+} for a generalization. 
The following statement provides a common 
generalization of the two formulas.

\begin{proposition} \label{prop:local-VE} 
For every triangulation $\Gamma$ of the simplex 
$2^V$ and every $E \subseteq V$,
\begin{equation} \label{eq:AKMS}
\ell_{V,E} (\Gamma, x) = \sum_{F \subseteq V}
\theta (\Gamma_F, x) \, 
d_{|V \sm F|,|V \sm (E \cup F)|} (x).
\end{equation} 
\end{proposition}

\begin{proof}
Using the defining equation~(\ref{eq:local-VE-def}) 
for $\ell_{V,E} (\Gamma, x)$ and
Equation~(\ref{eq:Ath22}), we get

\begin{align*}
\ell_{V,E} (\Gamma, x) & = 
\sum_{E \subseteq G \subseteq V} (-1)^{|V \sm G|}
h(\Gamma_G, x) \\ & = 
\sum_{E \subseteq G \subseteq V} (-1)^{|V \sm G|} 
\left( \, \sum_{F \subseteq G} \theta (\Gamma_F, x) 
\, A_{|G \sm F|} (x) \right) \\ & =  
\sum_{F \subseteq V} \theta (\Gamma_F, x) 
\left( \, \sum_{E \cup F \subseteq G \subseteq V} 
(-1)^{|V \sm G|} \, A_{|G \sm F|} (x) \right) \\ 
& = \sum_{F \subseteq V} \theta (\Gamma_F, x) \, 
d_{|V \sm F|,|V \sm (E \cup F)|} (x)
\end{align*}
and the proof follows.
\end{proof}

We recall that a triangulation $\Gamma$ of the 
simplex $2^V$ 
is called \emph{theta unimodal} (respectively, 
\emph{theta $\gamma$-positive}) \cite{Ath22+} if 
$\theta(\Gamma_F,x)$ is unimodal (respectively, 
$\gamma$-positive) for every $F \subseteq V$.
The barycentric subdivision of any regular cell 
decomposition of the simplex 
\cite[Theorem~4.6]{KMS19}, and the $r$-fold 
edgewise subdivision (for $r \ge n$), antiprism 
triangulation, interval triangulation and 
$r$-colored barycentric subdivision of any 
triangulation of the simplex $\sigma_n$ (by 
\cite[Corollary~3.9]{Ath22+}, combined with 
results of \cite[Section~7]{Ath22} 
\cite[Section~5]{ABK22} \cite[Section~5]{AT21}) 
are among those triangulations of $\sigma_n$ 
which are known to be theta $\gamma$-positive.

\begin{corollary} \label{cor:local-VE} 
Given any theta unimodal (respectively, theta
$\gamma$-positive) triangulation $\Gamma$ of 
the $(n-1)$-dimensional simplex $2^V$, the 
polynomials $x^n \ell_{V,E}(\Gamma,1/x)$ and 
$\sum_{E \subseteq G \subseteq V} h^\circ(\Gamma_G,
x)$ have a unimodal (respectively, $\gamma$-positive) 
symmetric decomposition with respect to $n$ for 
every $E \subseteq V$.
\end{corollary}

\begin{proof}
Since $\theta(\Gamma_F, x)$ is assumed to be 
unimodal (respectively, $\gamma$-positive), with 
center of symmetry $|F|/2$, for every $F \subseteq 
V$, the statement about $x^n \ell_{V,E}(\Gamma,1/x)$ 
follows directly from
Propositions~\ref{prop:dnk}~(b) 
and~\ref{prop:local-VE}. Applying the formula of 
Remark~\ref{rem:local-VE} (b) for $h(\Gamma_G,x) = 
x^{|G|} \, h^\circ(\Gamma,1/x)$ gives
\begin{align*}
\sum_{E \subseteq G \subseteq V} h^\circ(\Gamma_G,x) 
& = \sum_{E \subseteq G \subseteq V} \left( \, 
\sum_{E \subseteq F \subseteq G} x^{|G|} \, 
\ell_{F,E} (\Gamma_F, 1/x) \right) \\ & = 
\sum_{E \subseteq F \subseteq V} x^{|F|} \, 
\ell_{F,E} (\Gamma_F, 1/x) \left( \, 
\sum_{F \subseteq G \subseteq V} x^{|G \sm F|} \right)
\\ & = \sum_{E \subseteq F \subseteq V} x^{|F|} \, 
\ell_{F,E} (\Gamma_F, 1/x) \, (1+x)^{n-|F|}.
\end{align*}

\medskip
\noindent
The last sum has a unimodal (respectively, 
$\gamma$-positive) symmetric 
decomposition with respect to $n$ for every 
$E \subseteq V$, since $x^{|F|} \, \ell_{F,E} 
(\Gamma_F, 1/x)$ has such a decomposition with 
respect to $|F|$ for $E \subseteq F \subseteq V$.
\end{proof}

The following statement is the analogue of 
Theorem~\ref{thm:gamma} for the 
$\lL_\fF$ transformation.
\begin{corollary} \label{cor:gamma} 
Let $\Gamma$ be a uniform triangulation of the 
$(n-1)$-dimensional simplex with associated 
$f$-triangle $\fF$. 
\begin{itemize}
\itemsep=0pt
\item[(a)] 
If $\Gamma$ is theta unimodal, then 
$\iI_n (\lL_\fF(p(x)))$ has a unimodal symmetric 
decomposition with respect to $n$ for every 
$p(x) \in \pP_n[x]$.

\item[(b)] 
If $\Gamma$ is theta $\gamma$-positive, then 
$\iI_n (\lL_\fF(p(x)))$ has a $\gamma$-positive 
symmetric decomposition with respect to $n$ for 
every $p(x) \in \pP_n[x]$.

\end{itemize}
\end{corollary}

\begin{proof}
By Proposition~\ref{prop:geom}, for a uniform 
triangulation $\Gamma$ we have $\ell_{\fF,n,k}(x) 
= \ell_{V,E}(\Gamma,x)$ for every $(n-k)$-element
subset $E$ of $V$. Thus, $x^n \ell_{\fF,n,k}(1/x)$ 
is unimodal (respectively, $\gamma$-positive) for 
every $k \in \{0, 1,\dots,n\}$ by 
Corollary~\ref{cor:local-VE} and therefore so is 
\[ \iI_n (\lL_\fF(p(x))) = \sum_{k=0}^n c_k \, 
   \iI_n \, \lL_\fF \left( x^{n-k}(1+x)^k \right) 
	 = \sum_{k=0}^n c_k x^n \ell_{\fF,n,n-k}(1/x) \]
for every polynomial $p(x) = \sum_{k=0}^n c_k 
x^{n-k}(1+x)^k$ with $c_0, c_1,\dots,c_n \ge 0$.
\end{proof}

\begin{proof}[Proof of Theorem~\ref{thm:gamma}.] 
Similarly, it suffices 
to show that $\hH_\fF^\circ \left( x^{n-k}(1+x)^k 
\right)$ has a unimodal (respectively, 
$\gamma$-positive) symmetric decomposition with 
respect to $n$ for every $k \in \{0, 1\dots, n\}$. 
Indeed, 
\[ \hH_\fF^\circ \left( x^{n-k}(1+x)^k \right) = 
\sum_{i=0}^k {k \choose i} \hH_\fF^\circ (x^{n-i})
= \sum_{i=0}^k {k \choose i} h_\fF^\circ 
(\sigma_{n-i},x) \]
and the result follows from 
Corollary~\ref{cor:local-VE}.
\end{proof}

\section{On the real-rootedness of the $\hH_\fF^\circ$ 
         and $\lL_\fF$ transformations}
\label{sec:roots}

This section discusses possible generalizations
and analogues of \cite[Conjecture~1]{BJ22} in the 
framework of uniform triangulations. 
Let $\Gamma$ be a uniform triangulation of an 
$(n-1)$-dimensional simplex $\sigma_n$, with 
associated $f$-triangle $\fF$. Following 
\cite{Ath22, AT21}, we say that $\Gamma$ has the 
\emph{strong interlacing property} if

\begin{itemize}
\itemsep=0pt
\item[{\rm (i)}]
$h_\fF(\sigma_m, x)$ is real-rooted  for all 
$2 \le m < n$, 

\item[{\rm (ii)}]
$\theta_\fF(\sigma_m, x)$ is either identically
zero, or a real-rooted polynomial of degree $m-1$ 
with nonnegative coefficients which is interlaced 
by $h_\fF(\sigma_{m-1}, x)$,  for all $2 \le m \le 
n$.
\end{itemize}
These conditions imply strong real-rootedness 
properties for the $h$-polynomials of $\fF$-uniform 
triangulations of simplicial complexes and their 
symmetric decompositions \cite[Section~6]{Ath22} 
\cite[Section~4]{AT21}.

We recall that the polynomials $q_{\fF,n,k}(x)$ 
and $\ell_{\fF,n,k}(x)$ have been defined in 
Section~\ref{sec:geom} for $k \in 
\{0, 1,\dots,n\}$. By linearity of $\hH_\fF^\circ$
and $\lL_\fF$ we have 
\begin{align*} 
\iI_n (\hH_\fF^\circ (p(x))) & = \sum_{k=0}^n 
        c_k q_{\fF,n,k}(x) \\ 
\lL_\fF (p(x))) & = \sum_{k=0}^n 
        c_k \ell_{\fF,n,n-k}(x)
\end{align*}
for every polynomial $p(x) = \sum_{k=0}^n c_k 
x^{n-k} (1+x)^k$. Thus, the last claim in each 
part of the following conjecture is a consequence 
of the first.
\begin{conjecture} \label{conj:main} 
Let $\Gamma$ be a uniform triangulation of the 
$(n-1)$-dimensional simplex $2^V$ with associated 
$f$-triangle $\fF$. 
\begin{itemize}
\itemsep=0pt
\item[(a)] 
If $\Gamma$ has the strong interlacing property, 
then $(q_{\fF,n,k}(x))_{0 \le k \le n}$
is an interlacing sequence of real-rooted 
polynomials. In particular, $h(\Delta_\aA
(\Gamma),x)$ is real-rooted and so is 
$\hH_\fF^\circ(p(x))$ for every 
$p(x) \in \pP_n[x]$.

\item[(b)] 
If $\Gamma$ has the strong interlacing property, 
then $(\ell_{\fF,n,k}(x))_{0 \le k \le n}$
is an interlacing sequence of real-rooted 
polynomials. In particular, $\ell_V(\Gamma,x)$ 
is real-rooted and so is $\lL_\fF(p(x))$ for 
every $p(x) \in \pP_n[x]$.

\end{itemize}
\end{conjecture}

Conjecture~\ref{conj:main} unifies several results
and conjectures in the literature. For instance, 
the important special case of barycentric 
subdivision of Conjecture~\ref{conj:main} holds by 
Theorem~\ref{thm:main} and \cite[Theorem~3.6]{BS21}. 
The real-rootedness of $\ell_V(\Gamma,x)$ was 
proven for the $r$-colored barycentric subdivision 
in \cite[Section~3.3]{BS21} \cite[Section~5]{GS20} 
and was conjectured for the antiprism triangulation 
in \cite[Conjecture~5.8]{ABK22}. 
Among other instances, part (a) has been proven by 
the author for edgewise subdivisions (the proof will
appear elsewhere) and has been verified 
computationally for the interval triangulation 
for $n \le 10$ and for the antiprism triangulation 
for $n \le 15$. Part (b) has been verified 
for the interval and antiprism triangulations for 
$n \le 30$. 

\begin{remark} \label{rem:conj-main} \rm 
We conclude with some remarks on 
Conjecture~\ref{conj:main}.

(a) By replacing $h^\circ_\fF (\sigma_m,x)$
with $h_\fF (\sigma_m,x)$ in the definition of the 
$\hH^\circ_\fF$ transformation one gets the linear 
operator $\hH_\fF: \RR_n[x] \to \RR_n[x]$ defined 
by setting
\begin{equation}  
\hH_\fF(x^m) = h_\fF(\sigma_m,x)  
\label{eq:def-H}
\end{equation} 
for $m \in \{0, 1,\dots,n\}$. In the special case 
of barycentric subdivision one gets the linear 
transformation $\aA: \RR[x] \to \RR[x]$ for which 
$\aA(x^n) = A_n(x)$ for every $n \in \NN$. 
Computations suggest that 
Conjecture~\ref{conj:main} may still hold when 
the $\hH^\circ_\fF$ transformation is replaced 
by $\hH_\fF$.

(b) One may define the type $B$ analogues $\bB, 
\dD^B: \RR_n[x] \to \RR_n[x]$ of the Eulerian 
and derangement transformations, respectively, 
by setting $\bB(x^n) = B_n(x)$ and 
\[ \dD^B(x^n) = \sum_{i=0}^n \, (-1)^i 
   {n \choose i} B_{n-i}(x) \]
for $n \in \NN$, where $B_n(x)$ is the 
standard $n$th Eulerian polynomial of type $B$ 
\cite[Section~11.4]{Pet15}. Then, computational 
evidence again suggests that the polynomials 
$\bB(p(x))$ and $\dD^B(p(x))$ have only real 
roots for every $p(x) \in \pP[x]$. 

(c) More generally, given any sequence 
$(h_n(x))_{n \in \NN}$ of polynomials with 
nonnegative coefficients, one may define 
polynomials $h_{n,k}(x)$ and $\ell_{n,k}(x)$ 
for $k \in \{0, 1,\dots,n\}$ recursively by 
setting
\begin{align*} 
h_{n,k+1}(x) & = h_{n,k}(x) + xh_{n-1,k}(x) \\ 
\ell_{n,k+1}(x) & = \ell_{n,k}(x) - \ell_{n-1,k}(x)
\end{align*}
for $0 \le k < n$, or explicitly, by setting
\begin{align*} 
  h_{n,k}(x) & = \sum_{i=0}^k {k \choose i} x^i 
                    h_{n-i}(x) \\
  \ell_{n,k}(x) & = \sum_{i=0}^k (-1)^i  
                    {k \choose i} h_{n-i}(x) 							
\end{align*}						
for $k \in \{0, 1,\dots,n\}$. Under what conditions 
on the $h_n(x)$ are $(h_{n,k}(x))_{0 \le k \le n}$ 
and $(\ell_{n,k}(x))_{0 \le k \le n}$ interlacing 
sequences of real-rooted polynomials for every 
$n \in \NN$? 

Since $h_{n,0}(x) = h_n(x)$ and $h_{n,1}(x) = 
h_{n,0}(x) + x h_{n-1,0}(x)$, a necessary condition 
is that each $h_n(x)$ is real-rooted and it is 
reasonable to assume that $h_{n-1}(x)$ interlaces 
$h_n(x)$ for every $n \ge 1$. The example $h_n(x) 
= (1+x)^n$ shows that this condition is not 
sufficient even if each $h_n(x)$ is assumed to 
be symmetric of degree $n$ and center of symmetry
$n/2$. Indeed, if $h_n(x) = (1+x)^n$ for every 
$n \in \NN$, then $h_{n,k}(x) = 
(1+x)^{n-k} (1+2x)^k$ and $\ell_{n,k}(x) = x^k
(1+x)^{n-k}$ for all $n, k$ and the sequences 
$(h_{2,0}(x), h_{2,1}(x), h_{2,2}(x))$ and 
$(\ell_{2,0}(x), \ell_{2,1}(x), \ell_{2,2}(x))$ 
already fail to be interlacing. We note that 
$h_n(x) = (1+x)^n$ can be expressed as the 
$h$-polynomial of a \emph{nonuniform} 
triangulation of the $n$-dimensional simplex.

\end{remark}

\medskip
\noindent 
\textbf{Acknowledgments}.
The author wishes to thank Katerina
Kalampogia-Evangelinou for her help with all 
computations claimed in this article.

\end{document}